\documentclass[12pt]{amsart}

\usepackage{amsmath,graphics}
\usepackage{amsfonts,amssymb}
\usepackage{amscd}
\usepackage{mathrsfs}
\usepackage[all]{xy}
\usepackage{accents}
\usepackage{array}
\usepackage{tikz}
\usepackage{fancyvrb}
\usetikzlibrary{arrows,decorations.markings}

\theoremstyle{plain}
\newtheorem*{theorem*}{Theorem}
\newtheorem*{lemma*} {Lemma}
\newtheorem*{corollary*} {Corollary}
\newtheorem*{proposition*} {Proposition}
\newtheorem{theorem}{Theorem}[section]
\newtheorem{lemma}[theorem]{Lemma}
\newtheorem{corollary}[theorem]{Corollary}
\newtheorem{proposition}[theorem]{Proposition}

\theoremstyle{remark}
\newtheorem{remark}[theorem]{Remark}
\newtheorem{fact}[theorem]{Fact}
\newtheorem*{definition}{Definition}

\theoremstyle{definition}

\def \R {\mathbb{R}}

\def \Z {\mathbb{Z}}
\def \C {\mathbb{C}}
\def \P {\mathbb{P}}

\def \bn{\begin{enumerate}}
\def \en{\end{enumerate}}
\def \bdm{\begin{displaymath}}
\def \edm{\end{displaymath}}

\def \bp{\begin{proof}}
\def\ep{\end{proof}}

\def\vd{\vdots}
\def\dd{\ddots}
\def\U+e{(U^+)^e}

\def\be{\begin{equation}}
\def\ee{\end{equation}}
\def\mfg{\mathfrak{g}}

\def\mft{\mathfrak{t}}
\def\mfp{\mathfrak{p}}
\def\mfb{\mathfrak{b}}

\def\mo{\mathcal{O}}

\def\T{\triangle}

\begin{document}

\title [Conjecture $\mathcal{O}$  for $G/P$]{On the Conjecture $\mathcal{O}$ of GGI for  $G/P$}
\author{Daewoong Cheong}
\address{Korea Institute for Advanced Study \\
85 Hoegiro, Dongdaemun-gu\\
Seoul, 130-722, Korea}\email{daewoongc@kias.re.kr}
\author{Changzheng Li}
\address{Department of Mathematics, Sun Yat-Sen University, Guangzhou 510275, P.R. China;}
\address{ Center for Geometry and Physics, Institute for Basic Science (IBS), Pohang 790-784,   Korea}
\email{czli@ibs.re.kr}
\date{Dec. 1, 2014}

\begin{abstract}
In this paper, we show that general homogeneous manifolds $G/P$ satisfy Conjecture $\mathcal{O}$ of Galkin, Golyshev and Iritani which `underlies' Gamma conjectures I and II of them. Our main tools are the quantum Chevalley formula for $G/P$ and a theory on nonnegative matrices including Perron-Frobenius theorem.
\end{abstract}
\maketitle

\section{INTRODUCTION}
Let $X$ be a Fano manifold, i.e., a smooth projective variety whose anti-canonical line bundle is ample.  The quantum cohomlogy ring $H^\star(X,\C)$\footnote{We use this notation for the quantum cohomology ring with the multiplication $\star$, and without quantum variables.} of $X$ is a certain deformation of the classical cohomology ring $H^*(X,\C)$ (\S 2.4 below). For $\sigma\in H^\star(X,\C),$ define the quantum multiplication operator $[\sigma]$ on $H^\star(X,\C)$ by $[\sigma](\tau)=\sigma\star \tau$ for $\tau \in H^\star(X,\C)$, where $\star$ denotes the quantum product in $H^\star(X,\C)$. Let $\delta_0$ be the absolute value of a maximal modulus eigenvalue of the operator $[c_1(X)]$, where $c_1(X)$ denotes the first Chern class of the tangent bundle of $X.$ In \cite{GGI},
Galkin, Golyshev and Iritani say that $X$ satisfies Conjecture $\mathcal{O}$
if \bn
\item $\delta_0$ is an eigenvalue of $[c_1(X)]$.
\item The multiplicity of the eigenvalue $\delta_0$ is one.
\item If $\delta$ is an eigenvalue of  $[c_1(TX)]$ such that $|\delta|=\delta_0,$ then $\delta =\delta_0\xi$ for some $r$-th root of unity, where $r$ is the Fano index of $X.$

\en

\smallskip

In fact, in addition to Conjecture $\mathcal{O}$, Galkin, Golyshev and Iritani  proposed two more conjectures called Gamma conjectures I, II, which can be stated under the Conjecture $\mo.$  Let us briefly introduce Gamma conjectures I, II in order to explain how it underlies them. Consider the quantum connection of Dubrovin
$$\nabla_{z\partial_z}=z\frac{\partial}{\partial z}-\frac{1}{z}(c_1(X)\star)+\mu,$$ acting on $H^*(X,\C)\otimes \C[z,z^{-1}]$, where $\mu$ is the grading operator on $H^*(X)$ defined by $\mu(\tau)=(k-\frac{\mathrm{dim}X}{2})\tau$ for $\tau\in H^{2k}(X,\C).$
This has a regular singularity at $z=\infty$ and an irregular singularity at $z=0.$
 Flat sections near $z=\infty$ can be constructed through flat sections near $z=0$ classified by their exponential growth order, and they are put into correspondence with cohomology classes. To be precise, if $X$ satisfies  Conjecture $\mathcal{O}$, we can take a flat section $s_0(z)$ with the smallest asymptotics $ \sim e^{-\delta_0/z}$ as $z\rightarrow +0$ along $\R_{>0}$. We transport $s_0(z)$ to $z=\infty$ and identify the corresponding class $A_X$ called the principal asymptotic class of $X$. Then Gamma conjecture I states that the cohomolgy class $A_X$ is equal to the Gamma class $\hat{\Gamma}_X$.
Here $\hat{\Gamma}_X:=\prod_{i=1}^n\Gamma(1+\vartheta_i)\in H^*(X)$, where $\vartheta_i$ are the Chern roots of the tangent bundle $TX$ for $i=1,...,n$. Under further assumption of semisimplicity of the ring $H^\star(X),$ we can identify cohomology classes $A_\delta$ corresponding to each eigenvalue $\delta$ in similar way. The classes $A_\delta$ form a basis of $H^*(X,\C).$ Then Gamma conjecture II, a refinement of a part of Dubrovin's conjecture (\cite{Du}), states that there is an exceptional collection $\{E_\delta \hspace{0.04in}| \hspace{0.05in}\delta \hspace{0.05in} \mathrm{eignevalues\hspace{0.05in} of} \hspace{0.06in}[c_1(X)]\hspace{0.03in}\}$ of the derived category $D^b_{\mathrm{coh}}(X)$ such that for each $\delta$,   $$A_\delta= \hat{\Gamma}_X \mathrm{Ch}(E_\delta),$$ where $\mathrm{Ch}(E_\delta):=\sum_{k=0}^{\dim X}(2\pi \mathbf{i})^k \mathrm{ch}_k(E_\delta)$ is the modified Chern character. In this situation, Gamma conjecture I says that the exceptional object $\mo_X$ corresponds to $\delta_0.$
See \cite{GGI} and \cite{Du} for details on these materials.

\smallskip

As far as we know, the Conjecture $\mo$ has thus far been proved for the ordinary, Lagrangian and orthogonal Grassmannians. For the ordinary Grassmannian,
 Galkin, Golyshev and Iritani (\cite{GGI}) proved Conjecture $\mo$ together with Gamma conjectures I, II by using the quantum Satake of Golyshev and Manivel (\cite{GM}).  In fact we notice  that there were already two earlier papers proving Conjecture $\mo$ for the ordinary Grassmannian.  In $2006,$ Galkin and Golyshev (\cite{GG1}) gave a very short proof of Conjecture $\mo$ using a theorem of Seibert and Tian (\cite{ST}) and some elementary considerations. In $2003,$ Rietsch(\cite{Riet1})  gave a full description of eigenvalues and corresponding (simultaneous) eigenvectors of quantum multiplication operators for the Grassmannian, which actually proves Conjecture $\mo$, by using a result of Peterson   and some combinatorics.
Very recently, the first author  proved the Conjecture $\mo$ for Lagrangian and orthogonal Grassmannian (\cite{CH3}), following Rietsch (\cite{Riet1}).

 As for toric Fano  manifolds, Galkin, Golyshev and Iritani(\cite{GGI}) proved Gamma conjectures I, II modulo Conjecture $\mo$, and then Galkin (\cite{G1}) has made some progress on Conjecture $\mo$ by showing that the quantum cohomology ring of a toric Fano manifold contains a field as a direct summand.
 \smallskip

 It is natural to consider general homogeneous spaces $X=G/P$ as  next targets for Gamma conjectures. Indeed, here we prove Conjecture $\mo$ for homogeneous spaces as a first step into this project. A scheme of proof of Conjecture $\mo$ is to use the so-called quantum Chevalley formula which computes the multiplication $\sigma_1 \star \sigma_2$ of two basis elements  with $\sigma_1$ or $\sigma_2$ in $H^2(X,\Z)$, and a theory on nonnegative matrices including Perron-Frobenius theorem.

 To be precise, first note that the structure constants of the quantum product in $H^\star(X,\C)$ in the  basis of Schubert classes are three-point genus zero Gromov-Witten invariants. They  are actual counts of holomorphic spheres satisfying appropriate conditions, and are therefore nonnegative. Hence the matrix $M(X)$ of $[c_1(X)]$ with respect to this basis is a nonnegative matrix. Therefore once we prove that  $M(X)$ is irreducible, then by the celebrated Perron-Frobenius theorem (\S $3.1$ below), the conditions $(1)$ and $(2)$ are automatically satisfied. We remark that  the use of Perron-Frobenius theorem in the proof of $(1)$ and $(2)$ is due to Kaoru Ono (\cite[Remark 3.17]{GGI}).
However, the Perron-Frobenius theorem does not assert that the Fano index $r$ of $X$ is equal to the number $h$ of eigenvalues of maximal modulus, which is to be shown for the condition $(3).$ Towards this equality, it is already known that $r$ divides $h$ even for general Fano manifolds by \cite[Remark 3.1.3]{GGI}.   Then to show that conversely $h$ divides $r$, we bring a theory on directed graphs,  a disguise of  nonnegative matrices, into our situation and construct a certain number of cycles at a fixed vertex in the directed graph in question. The lengths of these cycles are used to show that $h$, in turn, divides $r$, and hence $r=h.$ This fact together with Proposition \ref{prop-angle} proves the condition $(3)$. Lastly, we point out that one of the advantages of our approach is that if   one of the eigenvalues  (of unnecessarily maximal modulus) of $[c_1(X)]$ is obtained, then one can recover other eigenvalues of the same modulus from the known eigenvalue by rotating it by a fixed angle depending on the Fano index of $X$ with the aid of Proposition \ref{prop-angle}.

{\it \textbf{Acknowledgements}.}  The authors would like to thank Sergey Galkin and Hiroshi Iritani for making comments on citations and pointing out some typos on an earlier version, and thank Leonardo C. Mihalcea for clarifying the reference on the nonvanishing of quantum products for general $G/P$. The
authors also thank the referee for various valuable comments.  The first author was
supported by National Researcher Program 2010-0020413 of NRF.
The second author was supported by IBS-R003-D1.

\section{Quantum cohomology of $G/P$}
We review some basic facts here.  Our readers can refer to \cite{Hu1, Hu2} for the details in the first subsection, and can refer to \cite{FW} and references therein  for  the rest.
\subsection{Notations}
Throughout this paper, $G$ denotes a complex, connected, semisimple, algebraic group, $B$ a fixed Borel subgroup, and $T$ a maximal torus in $B.$  As usual, $\mfg$, $\mfb$, and $\mft$ denote the Lie algebras of $G,B$ and $T$, respectively.  Denote the set of all roots by $R$.  Then we have a decomposition of root spaces $\mfg=\mft\oplus\bigoplus_{\alpha \in R}\mfg_\alpha.$ Let $\T$ be a set of simple roots  and $I$ be an indexing set for $\T$.
The parabolic subgroups $P$ of $G$ containing $B$ correspond to subsets $\triangle_P$ of $\triangle$. Let $I_P\subset I$ be the indexing subset for $\T_P,$ and $I^P=I\setminus I_P.$
Denote  the set of positive respectively negative roots relative to $B$ by $R^+$ respectively $R^-$. Let $R_P^+$ be the subset of positive roots which can be written as sums of roots in $\triangle_P$. Then the Lie algebra $\mfp$ has a decomposition of root spaces $\mfp=\mft\oplus\bigoplus_{\alpha \in R^+\sqcup (-R_P^+)}\mfg_\alpha$.

Let $W$ be the Wely group of $G,$ i.e., $W=N_G(T)/T.$ Then $W$ is generated by simple reflections $s_i$, $i\in I,$ where $s_{i}$ is a simple refection corresponding to $\alpha_i$.  For $u\in W,$ the length of $u$, denoted by $l(u),$ is defined to be the mimimum number of simple reflections whose product is $u.$ The Weyl group $W$ acts on $R$. Furthermore    for any $\gamma\in R^+$, there exist $w\in W$ and $i\in I$ such that $\gamma=w(\alpha_i)$. We then have a reflection $s_\gamma:=ws_iw^{-1}$ and the coroot $\gamma^\vee:=w(\alpha_i^\vee)$, which are  independent of the choices the choices of $w, i$. There is a unique  element of maximal length in $W$,   denoted by $w_0$. Then the opposite Borel subgroup $B^-$ is written as $B^-=w_0Bw_0.$

 Let $W_P$ be the subgroup of $W$ generated by the generators $s_{i}$ with $i \in I_P.$ Note that the generators $s_i$ with $i \in I_P$ are precisely ones such that $s_i \subset P$.  {Denote by $w_P$ the longest   element of $W_P.$  We will write $[u]$ for cosets $uW_P$ in $W/W_P$.  We will write $[u]$ for cosets $uW_P$ in $W/W_P$. It is well-known that each coset $[u]$ has a unique representative of minimal length. Let $W^P$ be the subset of $W$ consisting of such representatives in the cosets. Let $w_0^P$ be the minimal length representative in $[w_0]$, and hence $w_0^P$ is the   longest  element in  $W^P$. The length of $[u]$, denoted as $l([u])$, is defined to be the length of the minimal length representative in the coset $[u].$
 The dual  of $u \in W^P$, denoted $u^\vee$, is defined to be the minimal length representative of the coset $[w_0u]$.  Note that  $l(w_0^P)=\mathrm{dim}G/P$ and $l(u^\vee)=l(w_0^P)-l(u)$ for $u\in W^P$.
 Let $0^P$ be  the element of $W^P$ of minimum length. In fact, $0^P=\rm id$ is  the identity  of $W$, and so $l(0^P)=0.$

 \subsection{Cohomology}
 For convenience, throughout we will identify  the element $u\in W^P$ with  the element $[v]\in W/W_P$ if $u$ is the minimal length representative in $[v].$
For $u\in W^P,$ let $X(u)=\overline{BuP/P}$ be the Schubert variety corresponding to $u$ and $Y(u)=\overline{B^-uP/P}$ the opposite Schubert variety corresponding to $u$. Then $X(u)$ is a subvariety of $G/P$ of dimension $l(u)$ and $Y(u)$ is a subvariety of $G/P$ of codimension $l(u)$. Let $\sigma(u)$ respectively $\sigma_u$ be the cohomology class  $[X(u)]$ respectively $[Y(u)]$. Then we have the following classical results on $H^*(X)=H^*(X, \mathbb{Z})$, where $n=\dim_{\mathbb{C}}G/P$.
\bn
\item
For $u\in W^P$, $\sigma_u \in H^{2l(u)}(X)$,  $ \sigma(u)\in H^{2n-2l(u)}(X),$ and $\sigma_u=\sigma(u^\vee).$
\item
$H^*(X)=\bigoplus_{u\in W^P}\Z\sigma_u=\bigoplus_{u\in W^P}\Z\sigma(u).$ In particular, $H^2(X)=\bigoplus_{i\in I^P}\Z \sigma_{s_i}$, and $H^{2n-2}(X)=\bigoplus_{i \in I^P}\Z \sigma(s_i)$.
\item
$\int_{X}\sigma_u \cup \sigma_v=1$ if $v=u^\vee$, and $ 0$ otherwise.
\en

\subsection{Degrees}
Since the cohomology group $H^{2n-2}(X)$ can be canonically identified with the homology group   $H_2(X)$ by Poincar\'{e} duality, elements of $H^{2n-2}(X)$ may be referred to as  curve classes.  By a $degree$ $d$, we mean an effective class in  $H^{2n-2}(X)$, i.e., a nonnegative integral linear combination of the Schubert generators $\sigma(s_i)$ with $i\in I^P.$   A degree $d=\sum_{i\in I^P} d_i\sigma(s_i)$ can  be identified with $(d_{i})_{i\in I^P}$.

For $\alpha \in R^+$,  write $\alpha=\sum_{i\in I} m_{\alpha,\alpha_i}\alpha_i$ for some $  m_{\alpha,\alpha_i}\in \Z_{\geq0}$. Then we define the $degree$ of $\alpha$ as
$$d(\alpha)=\sum_{i\in I^P} m_{\alpha,\alpha_i}\frac{(\alpha_i,\alpha_i)}{(\alpha,\alpha)}\sigma(s_i).$$

Note that  $d(\alpha_i)=\sigma(s_i)$ if $i\in I^P$, and $d(\alpha_i)=0$ otherwise, since $ m_{\alpha_i,\alpha_j}=\delta_{i,j}$ for $i,j\in I$.
Let $h_{\alpha}=\frac{2\alpha}{(\alpha,\alpha)}$ and let  $\omega_{\alpha_i}$ be the fundamental weight corresponding to $\alpha_i,$ so that $h_{\alpha_i}$ and  $\omega_{\alpha_i}$ are dual bases for $i\in I$. Then $h_\alpha(\omega_{\alpha_i})=m_{\alpha, \alpha_i}(\alpha_i,\alpha_i)\slash(\alpha,\alpha),$ and hence we have $$d(\alpha)=\sum_{i\in I^P}h_\alpha(\omega_{\alpha_i})\sigma(s_i).$$
We notice that if $\alpha$ is a root, then the aforementioned $h_\alpha$ is identified with the coroot $\alpha^\vee$ via the Killing form $(\cdot, \cdot)$.

Set $$\rho_P=\frac{1}{2}\sum_{\alpha\in R^+\setminus R^+_P} \alpha ,\hspace{0.2in}  n_\alpha=4\frac{(\rho_P,\alpha)}{(\alpha,\alpha)}, \hspace{0.2in}\mathrm{and} \hspace{0.2in} n_i:=n_{\alpha_i} \mbox{ for }  i\in I^P.$$
\begin{lemma}[Lemma 3.5 of \cite{FW}]\label{lemma-Chernclass}
The first Chern class of $X=G/P$ is given by
$$c_1(X)= \sum_{i\in I^P}n_{i}\sigma_{s_i}=2\sum_{i\in I^P}h_{\alpha_i}(\rho_P)\sigma_{s_i}.$$
\end{lemma}
\begin{remark}\label{remark-Fano index} The coefficient $n_i$ is a positive integer for all $i\in I^P$. The positivity of  $n_i$ plays an important role in our proof together with the nonnegativity of the structure constants below (\S $2.4$). In the case   $P=B$, we have $R_B^+=\emptyset$ and $\rho_B=\sum_{i\in I} \omega_{\alpha_i}$.
\end{remark}
The Chern number   $\int_{X}c_1(X)\cup d(\alpha)$ of the degree $d(\alpha)$ is equal to $n_\alpha$. In particular, we have
  $n_i=\int_{[X(s_i)]}c_1(X)=d(\alpha_i)$
   for any $i\in I^P$. We notice  that   the Fano index of $X=G/P$ is given by
   $$r:=\mathrm{g.c.d}\hspace{0.03in}\{\hspace{0.02in}n_i\hspace{0.04in}|\hspace{0.04in} i\in I^P\}.$$

\subsection{Quantum cohomology of $G/P$} To define the quantum cohomology ring of $X$, we begin with Gromov-Witten invariants.
Given $u,v, w \in W^P$ and $d=\sum d_i\sigma(s_i)$ with $l(u)+l(v)+l(w)=\mathrm{dim}X+\int_Xc_1(X)\cdot d$, the \textit{three-pointed}, $genus$ $zero$ $Gromov$-$Witten$ $invariant$ associated with $u,v,w$ and $d$, denoted $c_{u,v}^{w,d}$, can be  defined as the number of morphisms $f:\P^1\rightarrow X$ of degree $d$ such that (fixed) general translates of $Y(u), Y(v)$ and $Y(w)$ pass through the three   points  $f(0),f(1)$ and $f(\infty)$,   respectively.

For each $i\in I^P,$ take a variable $q_i$, and let $\Z[q]$ be the polynomial ring with indeterminates $q_i, i\in I^P$.  We will regard  $\Z[q]$ as a graded $\Z$-algebra by assigning to $q_i$ the (complex) degree $n_i$. For a degree $d=\sum d_i\sigma(s_i),$ let $q^d$ stand for $\prod q_i^{d_i}$.
The quantum cohomology ring $qH^\star(X)$ of $X$,    as a $\mathbb{Z}[q]$-module, is defined to be $$qH^\star(X)=H^*(X)\otimes \Z[q].$$ The Schubert classes $\sigma_u$ with $u\in W^P$ form a $\Z[q]$-basis for $qH^*(X)$. The multiplication is defined as  \be \label{product}\sigma_u\star \sigma_v=\sum_d\sum_{w} c_{u,v}^{{w^\vee},d}\sigma_w,\ee
where the sums are taken over all $w\in W^P$ and degrees $d$ such that $l(u)+l(v)=l(w)+\int_Xc_1(X)\cdot d.$

 The quantum product of two general Schubert classes $\sigma_u$ and $ \sigma_v$ are far from completely understood. When either of them is in $H^2(X)$, then the so-called quantum Chevalley formula, due to Peterson \cite{Pert} and proved by Fulton and Woodward \cite{FW}, gives an explicit description of the coefficients in (\ref{product}).
\begin{proposition}[Quantum Chevalley formula]
For any $i\in I^P$ and $u\in W^P,$ the quantum product of $\sigma_{s_i}$ and $\sigma_u$ is given by
$$\sigma_{s_i}\star \sigma_u=\sum_\alpha h_{\alpha}(\omega_{\alpha_i})\sigma_v +\sum_\alpha q^{d(\alpha)}h_\alpha(\omega_{\alpha_i})\sigma_w,$$
 where the first sum is   over roots $\alpha \in R^+\setminus R^+_P$ for which $v=us_\alpha\in W^P$ satisfies  $l(v)=l(u)+1$, and the second sum is over roots $\alpha \in R^+\setminus R^+_P$ for which $w$ is the minimal length representative in $[us_{\alpha}]$ satisfying  $l(w)=l(u)+1-n_\alpha.$
\end{proposition}

\begin{remark}\label{remark-degree}
  Since $h_{\alpha_i}$ and $\omega_{\alpha_i}$ are dual bases for $i\in I$, by the very definition of $R^+_P,$ if $\alpha\in R_P^+, $  then $h_{\alpha}(\omega_{\alpha_i})=0$ for all $i\in I^P,$ and if $\alpha \in R^+\setminus R^+_P$, then there is an $i\in I^P$ such that $h_{\alpha}(\omega_{\alpha_i})\ne 0.$
\end{remark}

\begin{remark}
  We can also replace this parameterization set of the first sum by an   apparently larger one:  \textit {over roots $\alpha\in R^+$ for which the minimal length representative $v$ of $[us_\alpha]$ satisfies  $\ell(v)=\ell(u)+1$}. Indeed, the natural projection $G/B\to G/P$ induces an injective morphism $H^*(G/P)\hookrightarrow H^*(G/B)$ of algebras, sending a Schubert class $\sigma_u^P \,\,(u\in W^P)$ in $H^*(G/P)$ to the Schubert class $\sigma^B_u$ in $H^*(G/B)$ labeled by the same $u$. The new parameterization set of the first sum gives the Chevalley formula for $H^*(G/B)$. Due to   the injective morphism,   the coefficient $h_\alpha(\omega_{\alpha_i})$ is nonzero only if $v=us_\alpha$ itself belongs to $W^P$ and $\alpha\notin R_P^+$.
\end{remark}
\begin{corollary}
 For any $i\in I^P$, we have $2 \leq n_i \leq   \mathrm{dim}  X+1$.
\end{corollary}
\begin{proof}
   It follows from the definition of $n_i$ that $n_i\geq 2$. Since $\ell(w_0^P)=\dim G/P$, the cup product    $\sigma_{s_i}\cup \sigma_{w_0^P}=0$ vanishes.  On the other hand,    the quantum product  of Schubert classes never vanishes (following from \cite[Theorem 9.1]{FW}). Therefore by the quantum Chevalley formula, the expansion of $\sigma_{s_i}\star \sigma_{w_0^P}$ contains a class $q^{d(\alpha)}\sigma_{w}$ for some $\alpha\in R^+\setminus R_P^+$ which has a positive coefficient at the simple root $\alpha_i$. It follows that
      $n_i=\deg (q_i)\leq \deg (q^{d(\alpha)})+\ell(w)=\ell(s_i)+\ell(w_0^P)=1+\dim X$.
\end{proof}
 The   quantum cohomology ring of $X$ can also be defined without using the quantum variables $q_i$  ($i\in I^P$), which is  denoted as  $H^\star(X,\C)$.
 In our language, the ring  $H^\star(X,\C)$ is identified with
the specialization of $qH^\star(X,\C)$ at $q_i=1$ for all $i$, i.e., $$H^\star(X,\C)=qH^\star(X,\C)/<q_i-1\hspace{0.04in}|\hspace{0.04in}i\in I^P>.$$ Note that  $H^\star(X,\C)$ is a finite dimensional vector space over $\C$, while $qH^\star(X,\C)$ is not over $\C$, but over $\C[q].$

 \section{Nonnegative matrices}
In this section, we review Perron-Frobenius theory on nonnegative matrices and some related results which will be used later. Details on these materials can be found in \cite{Minc} and \cite{BP}.
\subsection{Irreducible matrices}
\begin{definition}
A nonnegative matrix $M$ is said to be  $cogredient$ to a matrix $M^\prime$ if there is a permutation matrix $P$ such that $M=P^TM^\prime P$.

A nonnegative matrix $M$ is called  $reducible$ if it is cogredient to a matrix in the form
\be \label{Mat}
M^\prime=\left[\begin{array}{cc}
A&B\\
0&D
\end{array}\right],
\ee  where $A,D$ are square submatrices.
 If it is not reducible, then $M$ is called $irreducible$.
\end{definition}
\begin{remark}\label{remark-basis}
\bn
\item
 Note that if $V$ is a vector space with an ordered basis $\mathcal{B}=\{v_1,...,v_m\}$ and $T$ is an operator on $V,$ then the matrix $[T]_{\mathcal{B}}$ of $T$ with respect to the basis $\mathcal{B}$ is reducible if and only if there is a nontrivial proper coordinate subspace invariant under $T$, equivalently there is an ordered basis $\mathcal{B}^\prime$ with respect to which $[T]_{\mathcal{B}^\prime}$ is in the form (\ref{Mat}), where $\mathcal{B}^\prime$ is obtained from $\mathcal{B}$ by reordering elements of $\mathcal{B}.$
\item
 Suppose $[T]_\mathcal{B}$ is reducible. Let $V_0$ denote a  nontrivial proper coordinate subspace of $V$ invariant under $T$. We point out that  if $V_0$ contains a basis element $v_i\in \mathcal{B}$, then $V_0$ contains all basis elements $v_j\in \mathcal{B}$ such that the coefficient $b_{ji}$ of $v_j$ is nonzero in $T(v_i)=\sum_{k=1}^m b_{ki} v_k.$ More generally, suppose $T=\sum_{i=1}^lc_iT_i$ for some positive numbers $c_i$ and operators $T_i$ with $[T_i]_\mathcal{B}$ nonnegative. Then the coefficient $b_{ji}$ is nonzero if and only if there exists a $1\leq p \leq l$  such that the coefficient $b_{ji}^p$ of $v_j$ is nonzero in  $T_p(v_i)=\sum_{k=1}^m b_{ki}^p v_k$.
\en

\end{remark}

The next two propositions are due to Perron \cite{Perr} and Frobenius \cite{Frob}.
\begin{proposition}[See e.g.  Theorem 1.4 of Chapter 2 of \cite{BP}]\label{Perron-Frob} Let $M$ be an irreduclble matrix. Then
$M$ has a real positive eigenvalue $\delta_0$ of multiplicity one such that $$\delta_0\geq |\delta|$$ for any eigenvalue $\delta$ of $M.$ Furthermore, $M$ has   a  positive eigenvector corresponding to $\delta_0.$
\end{proposition}

\begin{definition} For an irreducible matrix $M,$ we define the $index$  $of$ $imprimitivity$ of $M$, denoted as $h(M)$, to be the number of eigenvalues of maximal modulus. If $h(M)=1$, then $M$ is said to be $primitive$; otherwise, it is $imprimitive$.
\end{definition}

If $M$ is an irreducible matrix, then eigenvalues of the same modulus are completely determined by one of them.
\begin{proposition}[See e.g. Theorem 2.20 of Chapter 2 of \cite{BP}]\label{prop-angle}
Let $M$ be an irreducible matrix with $h(M)=h.$ Then the eigenvalues of $M$ of modulus $\delta_0$ are  all of multiplicity one,   given by the distinct roots of $\lambda^h-\delta_0^h=0$.
 {Moreover, the set of eigenvalues of $M$ is invariant under rotation by $\frac{2\pi}{h}$}.
\end{proposition}

\begin{definition}
A matrix in the form
\be \label{superblock}
\left[\begin{array}{cccccc}
0&A_{12}&0&\cdots&0&0\\
0&0&A_{2,3}&\cdots&0&0\\
\vd&&&\dd&0&\vd\\
0&0&&\cdots&0&A_{k-1,k}\\
A_{k1}&0&&\cdots&&0
\end{array}\right]
\ee
is said to be $in$ $the$ $superdiagonal$ $(m_1,m_2,...,m_k)$-$block$ $form$ if the block $A_{i,i+1}$ is a $(m_i\times m_{i+1})$ matrix for  $i=1,...,k-1$, and $A_{k,1}$ is a $(m_k\times m_1)$ matrix.
\end{definition}
If an irreducible matrix $M$ is cogredient to a matrix $M^\prime$ in the form (\ref{superblock}), much spectral information of $M$ can be read off from $M^\prime$.
Among them, first comes the index of imprimitivity.
\begin{proposition}[\protect\cite{Minc00}; see e.g. Theorem 4.1 of Chapter 3 of \cite{Minc}]\label{Pro divide}
Let $M$ be an irreducible matrix with $h(M)=h$. Then
$M$ is cogredient to a matrix in the form (\ref{superblock}) such that all the $k$ blocks $A_{1,2},\cdots, A_{k-1, k}, A_{k,1}$ are nonzero if and only if $k$ divides $h.$

\end{proposition}

\subsection{Directed graphs}
When we deal with spectral properties of nonnegative matrices, mostly we are only interested in the zero pattern of their entries.  One of ways of encoding this pattern is through the so-called directed graph. We list a multiple of definitions related with  directed graphs.
\begin{definition}
\bn
\item
A $directed$ $graph$ $D$ consists of data $(\textrm{Ver},\textrm{Arc})$, where $\textrm{Ver}$ is a set and $\textrm{Arc}$ is a binary relation on $\textrm{Ver}$, i.e., a subset of $\textrm{Ver}\times \textrm{Ver}.$ Elements of $\textrm{Ver}$ are called $vertices$ and elements of $\textrm{Arc}$ are called $arcs$. For convenience, we assume that  $\mathrm{Ver}=\{v_1,...,v_m\}.$
\item
A sequence of arcs $(v_{i_1},v_{i_2}),(v_{i_2},v_{i_3}),(v_{i_3},v_{i_4}),...,(v_{i_{k-1}},v_{i_k})$ in $D$ is called a $path$ $from$  $v_{i_1}$ $to$ $v_{i_k}$ which we will denote by $\mathrm{PATH}(v_{i_1}:v_{i_k})$, or simply $\mathrm{PATH}(i_1,i_l)$.
The $length$ of a path is the number of arcs in the sequence. A path of length $k$ from a vertex to itself is called a $cycle$ of length $k.$
\item
The $adjacency$ $matrix$ \footnote{Our definition of the adjacency matrix may be slightly different from ones in some literature.  The one in \cite{Minc} (p.$77$ ) is the transpose of ours. Our definition is a bit more intuitive in our situation.} of $D$, denoted $A=A(D)=(a_{i,j}),$ is an $m\times m$ square matrix with entries $0$ or $1$ defined by
\begin{displaymath} a_{i,j}= \left\{\begin{array}{cc}
1& \mathrm{if}\hspace{0.07in} (v_j,v_i)\in \mathrm{Arc}, \\
0,& \mathrm{otherwise.}
\end{array}\right.
\end{displaymath}

\item
A directed graph $D$ is said to be $associated$ with a  nonnegative matrix $M$ if the adjacency matrix $A(D)$ has the same zero pattern as $M.$
\item
A directed graph is $strongly$ $connected$ if for any ordered pair $(v_i,v_j)$ with $i\ne j,$
there is a path from $v_i$ to $v_j$
\item
Let $D$ be a strongly connected directed graph. The $index$ $of$ $imprimitivity$ of $D$, denoted as $h(D)$, is defined to be the g.c.d of  lengths of all cycles in $D$
\en
\end{definition}

\begin{remark}\label{remark-graph} A directed graph $D$ can be visualized by a diagram in which an arc $(v_i,v_j)$ is represented by a directed line going from $v_i$ to $v_j.$
The diagram associated to $D$ is also referred to as a directed graph.
 \end{remark}
 \begin{fact}\label{fact-graph}
   Let $V$ be a vector space with an ordered basis $\mathcal{B}=\{v_1,...,v_m\}$, and, for $i=1,...,l$, let $T_i$ be an operator on $V$ with $[T_i]_\mathcal{B}$ nonnegative. Given positive real numbers $c_1, ...,c_l$, we let $T:=\sum_{i=1}^lc_iT_i$, and   can associate to $T$ a directed graph $$D(T:\mathcal{B})=(\mathrm{Ver}(T:\mathcal{B}), \mathrm{Arc}(T:\mathcal{B})).$$
Here  $\mathrm{Ver}(T:\mathcal{B})=\{v_1,...,v_m\},$ and we define a relation $\mathrm{Arc}(T:\mathcal{B})$ on $\mathrm{Ver}(T:\mathcal{B})$ by $$(v_i,v_j)\in \mathrm{Arc}(T:\mathcal{B})\Leftrightarrow b_{ji}\ne 0 \hspace{0.07in} \mathrm{in} \hspace{0.07in} T(v_i)=\sum_{k=1}^m {b_{ki}}v_k,$$

$$\Leftrightarrow \exists \hspace{0.07in} p, 1\leq p\leq l, \hspace{0.04in} \mathrm{such\hspace{0.04in} that} \hspace{0.07in}b_{ji}^p\ne 0 \hspace{0.07in} \mathrm{in} \hspace{0.07in} T_p(v_i)=\sum_{k=1}^m {b_{ki}^p}v_k.$$
\noindent The matrix  $[T]_{\mathcal{B}}$ has the same zero pattern as the adjacency matrix $A(D(T:\mathcal{B}))$, and hence the directed graph $D(T:\mathcal{B})$ is associated with the matrix $[T]_{\mathcal{B}}$.  In particular,  to any nonnegative matrix $M,$ we can associate the directed graph $D(M):=D(M,\mathcal{B})$ by taking $\mathcal{B}$ to be the standard basis.
 \end{fact}
The next proposition compares the properties of nonnegative matrices and their assosiated directed graphs; see e.g. Theorems 3.2 and 3.3 of Chapter 4 of \cite{Minc}.

\begin{proposition}\label{equal} Let $M$ be a nonnegative matrix.
\bn
\item
$M$ is irreducible if and only if the associated directed graph $D(M)$ is strongly connected.
\item If $M$ is irreducible, then
the index $h(M)$ of imprimitivity of $M$ is equal to the index $h(D(M))$ of imprimitivity of the associated directed graph $D(M)$.
\en
\end{proposition}

\section{Main result}
The quantum cohomology ring $H^\star(X,\C)$ is a finite dimensional complex vector space with the Schubert basis $\mathcal{S}$ consisting of Schubert classes $\sigma_u$ for $u\in W^P$. Arrange  elements of $\mathcal{S}$ linearly once and for all to make $\mathcal{S}$ into an ordered basis. We will denote this ordered basis by $\mathcal{S},$ too.

The next lemma is a known fact, and can be found  for instance in \cite[Lemma 2.7]{LL01}.
\begin{lemma}\label{reducedword}
   Let $w\in W^P$, and take any reduced decomposition $w=s_{i_1}s_{i_2}\cdots s_{i_l}$ where $\ell=\ell(w)$.  Then we have $v:=s_{i_2}\cdots s_{i_l}\in W^P$ and $v^{-1}(\alpha_{i_1})\in R^+\setminus R_P^+$.
\end{lemma}
\begin{proposition}\label{pro-irreducible}
Let $M(X)$ be a matrix of the operator $[c_1(X)]$ on  $H^\star(X,\C)$  with respect to $\mathcal{S}.$ Then
$M(X)$  is an irreducible, nonnegative, integral matrix.
\end{proposition}

\begin{proof} As we can see in Lemma \ref{lemma-Chernclass}, the first Chern class $c_1(X)$ of the homogeneous variety $X=G/P$ is a nonnegative integral combination of Schubert divisor classes. By the quantum Chevalley formula, each  coefficient in the quantum multiplication $\sigma_{s_i}\star \sigma_u$, being of the form $h_\alpha(\omega_{\alpha_i})$, is again a  nonnegative integer. Moreover, $H^\star(X, \mathbb{C})$ is obtained from $qH^*(X, \mathbb{C})$ by taking specializations at $q_i=1$ for all $i$. It follows that $M(X)$ is a nonnegative integral matrix.

 Suppose $M=M(X)$ is reducible. Then, by definition, there exists a permutation matrix $P$ such that $M=P^TM'P$ for a block-upper triangular matrix $M'$. It follows that
  $M^m=P^{T}(M')^mP$ is  reducible for any $m\in\mathbb{Z}_{\geq0}$, and so is
  $\sum_{m=0}^n M^m$ (recall $n=\dim_{\mathbb{C}} G/P$). 
  Let $V$ be the nontrivial proper coordinate subspace of $H^\star(X)$ which is invariant under the operator    $T:=\sum_{m=0}^n([c_1(X)])^m$.
 By the quantum Chevalley formula, for any $u\in W^P$,     the coefficient $b_u(q)$ of   $\sigma_u$ in the quantum product $c_1(X)\star \cdots\star c_1(X)$ belongs to $\mathbb{Z}_{\geq 0}[q]$. Recall  $q=(q_i)_{i\in I^P}$, and denote $\vec 1:=(1)_{i\in I^P}$, $\vec 0:=(0)_{i\in I^P}$.   First we claim that
    \begin{enumerate}
     \item for any $w\in W^P$,  the coefficient  $b_w(q)\in \mathbb{Z}_{\geq 0}[q]$ of $\sigma_w$   in   the quantum multiplication  $c_1^{\ell(w)}$  of $c_1:=c_1(X)$ has
           a positive constant term  $b_{w}(\vec 0)>0$;  therefore  the operator $T$ must be  of the form
                 $T= \sum_{w\in W^{P}}a_w( {\vec 1})[\sigma_w]$,
    where $a_w({q})\in\mathbb{Z}_{\geq 0}[ {q}]$ with $a_w( {\vec 0})>0$, and   $[\sigma_w]$ denotes the operator on $H^\star(X, \mathbb{C})$ defined by
            $[\sigma_w](\beta)=\sigma_w\star\beta|_{q=\vec 1}$;
     \item $\sigma_{w_0^P}\in V$.
    \end{enumerate}
  To prove  $(1),$  we proceed by induction. If $l(w)=0,$ then this is trivial since $w=w_{0^P}$.   If $\ell(w)=l$, then
   we take a reduced decomposition $w=s_{i_1}\cdots s_{i_l}$. By Lemma \ref{reducedword}, we have
       $v:=s_{i_2}\cdots s_{i_l}\in W^P$ and $\gamma:=v^{-1}(\alpha_{i_1})\in R^+\setminus R_P^+$, which implies that   $h_\gamma(\omega_j)\neq 0$ for some $j\in I^P$.  Thus
       $\sigma_w$ occurs in the classical part of $\sigma_{s_j}\star \sigma_v$ with a positive coefficient by the Chevalley formula, and so does in $c_1\star \sigma_v$. Clearly,  $\ell(v)=\ell(w)-1$. Therefore     the coefficient $b_v(q)\in \mathbb{Z}_{\geq 0}[q]$ of $\sigma_v$  in $c_1^{\ell(v)}$ satisfies $b_{v}(\vec 0)>0$   by the induction hypothesis, and consequently   $\sigma_w$  occurs in $c_1^{\ell(w)}$ with the same property due to the positivity of quantum multiplication of Schubert classes.
      To prove (2), we take $w\in W^P$ such that $\sigma_w\in V$. By  item (3) in section 2.2,   $\sigma_{w_0^P}$ occurs in the classical part of $[\sigma_{w}](\sigma_{w^\vee})=\sigma_w\star \sigma_{w^\vee}$, and hence     it occurs in $T(\sigma_{w^\vee})$.

  Now choose $u\in W^P$ and $d\in H^{2n-2}(X, \mathbb{Z})$ such that $c_{w_0^P, w_0^P}^{u^\vee, d}>0$. Such elements always exist since the quantum product of two Schubert classes never vanishes \cite[Theorem 9.1]{FW}. But since we have $c_{u^\vee, w_0^P}^{{(0^P)^\vee}, d}  =c_{w_0^P, w_0^P}^{u^\vee, d}>0$ by the symmetry of Gromov-Witten invariants,  the basis element $\sigma_{0^P}$ occurs in the quantum product
  $\sigma_{u^\vee}\star \sigma_{w_0^P}$ and hence in $T(\sigma_{w_0^P})$. Thus   $\sigma_{0^P}\in V$, and  hence for any $w\in W^P$, $\sigma_w$ occurs in  $T(\sigma_{0^P})$ by noting   $[\sigma_w](\sigma_{0^P})=\sigma_w\star \sigma_{0^P}=\sigma_w$.    This implies $V=H^\star(X)$, contradicting the hypothesis    $V\subsetneqq H^\star(X)$.  Therefore, the matrix $M(X)$ is irreducible.
 \end{proof}

 \begin{remark}
\bn
\item
We note that the proof of Proposition \ref{pro-irreducible} works for any quantum multiplication operators of the form $[\sigma]=\sum_{i\in I^P}a_i [\sigma_{s_i}]$ for positive real numbers $a_i.$
\item
We mention that a general idea of proving the irreducibility of $M(X)$ was taken more or less from Lemma $9.3$ (p. $384$) in \cite{Riet2}, where Rietsch used some Peterson's results to show the irreducibility of the matrix  $[\sigma]_\mathcal{S}$  for a flag manifold $X$ of type $A$, where $\sigma=\sum_{w\in W^P}\sigma_{w}\in H^\star(X,\C)$.
\en
\end{remark}

Recall that $r=\mathrm {g.c.d}\hspace{0.03in}\{\hspace{0.02in}n_i \hspace{0.02in}|  \hspace{0.04in}i\in I^P\}$ is the Fano index of $X$, and $\mathcal{S}$ is an ordered basis of $H^\star(X)$. The order on $\mathcal{S}$ induces a linear order on the index set $W^P$, denoted as $\prec$. Let us make a partition on the set $W^P $ into $r$ subsets.
For each $0\leq a\leq r-1$, let  $W^P(a)$ be the subset of $W^P$ consisting of elements $u$ with $l(u)\equiv a \hspace{0.1in} \mathrm{mod}\hspace{0.07in} r$, and we assign to each  $W^P(a)$ the weight $a$.  Now we define a new linear order $\prec_q$ on $W^P$ as

\begin{displaymath} u\prec_q v \Leftrightarrow  \left\{\begin{array}{c}
u\in W^P(a),\hspace{0.07in} v\in W^P(b),\hspace{0.07in}\textrm{and}\hspace{0.07in} a>b, \\
u,v \in W^P(a), \hspace{0.07in}\mathrm{and}\hspace{0.07in} u\prec v.
\end{array}\right.
\end{displaymath}
\noindent The order $\prec_q$ on $W^P$ naturally makes  the Schubert basis for $H^\star(X)$ into an ordered basis, denoted as $\mathcal{S}_q$, in the way that the basis elements $\sigma_u$ for  $u\in W^P(r-1)$ come first, $\sigma_u$ for  $u\in W^P(r-2)$ second, and so on.

 \begin{lemma}\label{lemma-divide}
 The matrix $M(X)_q$ of the operator $[c_1(X)]$ on $H^\star(X)$ with respect to the ordered basis $\mathcal{S}_q$ is in the superdiagonal $(m_1,m_2,...,m_r)$-block form (\ref{superblock}) with $k=r$, where $m_i:=|W^P(r-i-1)|$ for $i=1,...,r,$  and $W^P(-1):=W^P(r-1)$.
 \end{lemma}
 \begin{proof}
 It is obvious that the blocks $A_{1,2},\cdots, A_{r-1, r}$, $A_{r,1}$ are nonzero. Clearly,
 if a basis element $\sigma_v$ appears with a nonzero coefficient in the expansion of  $[c_1(X)] (\sigma_u)$ in the basis $\mathcal{S}_q$, then  by the degree condition of quantum multiplication  we have $$l(v)\equiv l(u)+1 \hspace{0.1in} \mathrm{mod}\hspace{0.07in} r.$$ This proves the lemma.
 \end{proof}

The next lemma  holds for general Fano manifolds (\cite[Remark 3.1.3]{GGI}). In the case of homogeneous spaces $X=G/P$, it also follows immediately from Lemma \ref{lemma-divide} and Proposition \ref{Pro divide}, by noting that $M(X)$ is cogredient to $M(X)_q$.
 \begin{lemma}\label{coro-divide-p}
 The Fano index $r$ of $X$ divides the index of imprimitivity $h(M(X))$ of $M(X).$
 \end{lemma}

 \begin{definition}For a homogeneous space $X=G/P,$
 define the directed graph $D(X)$ as $$D(X)=D([c_1(X)],\mathcal{S}),$$
 where $D([c_1(X)],\mathcal{S})$ was defined in Fact \ref{fact-graph}.
 \end{definition}

 Since $M(X)$ is irreducible, then its associated directed graph $D(X)$ is strongly connected by Proposition \ref{equal}, and hence it makes sense to consider the index $h(D(X))$ of imprimitivity of $D(X).$ We will show that $h(D(X))$ divides $r$.
 \smallskip

 Let  $Q^\vee$ (resp. $Q_P^\vee$) denote  the coroot (sub)lattice $\oplus_{\alpha\in \Delta}\mathbb{Z}\alpha^\vee$ (resp. $\oplus_{\alpha\in \Delta_P}\mathbb{Z}\alpha^\vee$). Then we have the natural identifications $$Q^\vee/Q_P^\vee=H_2(G/P,\Z)=H^{2n-2}(G/P,\Z).$$

 \begin{lemma}\label{uniqueness}
 For any $\lambda_P\in Q^\vee/Q_P^\vee$,  there exists a unique $\lambda_B\in Q^\vee$ such that $\lambda_P=\lambda_B+Q_P^\vee$ and
                $\langle \alpha, \lambda_B\rangle  \in \{0, -1\}$ for all $\alpha\in R^+_P$ with respect to the natural pairing $\langle\cdot, \cdot\rangle: \mathfrak{h}^*\times \mathfrak{h}\to \mathbb{C}$.
 \end{lemma}
 Given $\lambda_B$ in the above lemma, we let $P^\prime$ be the parabolic subgroup such that  $\Delta_{P'}:=\{\alpha\in \Delta_P~|~\langle  \alpha, \lambda_B\rangle =0\}$. Recall that for a parabolic subgroup $P,$ $w_P$ denotes the longest  element of $W_P.$
 Let us record the following comparison formula which will be used to prove our key lemma. Both Lemma \ref{uniqueness} and the next proposition are due to Peterson \cite{Pert}, proved  by Woodward \cite{Wood}. Here we are following the equivalent formulation given in    \cite[Theorem 10.13]{LS}.
 \begin{proposition}\label{propcomparison}
   For any $u, v, w\in W^P$, we have
 $$c_{u,v}^{w^\vee, \lambda_P }=c_{u, v}^{(ww_Pw_{P'})^\vee,   \lambda_B},$$
among the structure constants for $qH^\star(G/P)$ and $qH^\star(G/B)$ respectively.
\end{proposition}

We will need the following lemma. Here we provide an involved combinatorial proof, while a  geometric proof by carefully studying the space of lines in $G/P$ (cf.
\cite{strickland:lines} and \cite{li.mihalcea:lines}) is quite desirable and will probably be much simpler.
 \begin{lemma}\label{qiterm}
     For any $i\in I^P$,   there exists $u\in W^P$ of length $\ell(u)=n_i-1$ such that $\sigma_{s_i}\star \sigma_u$ contains the $q_i\sigma_{0^P}$-term with coefficient $1$.
 \end{lemma}

\begin{proof}
Let $\lambda_P=\alpha_i^\vee+Q_P^\vee$. The unique lifting $\lambda_B\in Q^\vee$ is a priori, not necessarily a coroot. Nevertheless,  let us
first assume the next  claims hold.
\begin{enumerate}
  \item[a)] There  exists    $\gamma\in R^+$ of length $\ell(s_\gamma)=2 \langle \rho, \gamma^\vee\rangle-1$ such that $\lambda_B=\gamma^\vee$ is the coroot of $\gamma$, where   $\rho:=\rho_B$ equals the sum of fundamental weights of $(G, \Delta)$.
  \item[b)] $\ell(w_Pw_{P'}s_\gamma)=\ell(w_Pw_{P'})+\ell(s_\gamma)$.
\end{enumerate}
 We then  set $u:= w_Pw_{P'}s_\gamma$.  Take any $\alpha\in \Delta_P$. If $\langle \alpha, \gamma^\vee\rangle=0$, then $s_\gamma(\alpha)=\alpha\in \Delta_{P'}$, and hence
       $ u(\alpha)=   w_Pw_{P'}(\alpha)\in R^+$, by noting    $ w_Pw_{P'}\in W^{P'}$.
   If $\langle \alpha, \gamma^\vee\rangle\neq 0$, then it  equals   $-1$ due to the property of $\lambda_B=\gamma^\vee$.
      On one hand, $u(\alpha)\in R$ is a root and hence must be either purely nonpositive or purely nonnegative combinations of the simple roots; on the other hand,
          $u(\alpha)= w_Pw_{P'} (\alpha+  \gamma)=w_Pw_{P'} (\alpha_i)+  w_Pw_{P'} (\alpha+  \gamma-\alpha_i)\in \alpha_i+\sum_{\beta\in \Delta_P}\mathbb{Z}\beta$ (since $\alpha+\gamma-\alpha_i\in  \sum_{\beta\in \Delta_P}\mathbb{Z}\beta, w_Pw_{P'}\in W_P$ and $s_j(\alpha_i)\in \alpha_i+\sum_{\beta\in \Delta_P}\mathbb{Z}\beta$ for all $\alpha_j\in \Delta_P$). Therefore  $u(\alpha)\in R^+$.   It follows that $u\in W^P$ (by a characterization of minimal length representatives e.g. in \cite[section 5.4]{Hu1}).
Consequently, we have $c_{s_i, u}^{(0^P)^\vee, \lambda_P}=c_{s_i, u}^{(w_Pw_{P'})^\vee, \lambda_B}= c_{s_i, u}^{(us_\gamma)^\vee, \gamma^\vee}=1$ by Proposition \ref{propcomparison} and the quantum Chevalley formula of $qH^\star(G/B)$.

   It remains to prove the claims.
  Denote by  $\Delta(\alpha_i)$ the connected component of $\alpha_i$ in the Dynkin sub-diagram $\{\alpha_i\}\cup \Delta_P$ of $\Delta$, and by
  $\{\Pi_1, \cdots, \Pi_k\}$   the connected components of Dynkin diagram of $\Delta(\alpha_i)\setminus\{\alpha_i\}$.
  Let $\omega_{(j)}$ (resp. $\omega'_{(j)}$) denote the longest   element in the Weyl subgroup of $\Pi_j$ (resp. $\Pi_j \cap \Delta_{P'}$), and set $w_j:=\omega_{(j)}\omega'_{(j)}$. It follows that $w_iw_j=w_jw_i$ for any $i, j$, and that $w_Pw_{P'}=w_1\cdots w_k$.
 We notice  that all    the possible (non-empty)
    $\Pi_j$ have been listed in \cite[Table 3]{LL01} and \cite[Table 2.1]{Li}. By the classification, one can see $\lambda_B=\alpha_i^\vee$ in most of the cases, and will be able to  precisely describe       $\gamma$ and $ w_Pw_{P'}$  in the all the remaining cases. Therefore  the claims can be shown  by direct calculations. To be precise, we give the details as follows.

   For every $j$, we denote by  $ Q_{(j)}^\vee:= \sum_{\alpha\in \Pi_j}\mathbb{Z}\alpha^\vee$ the coroot sublattice of $\Pi_j$, by $R_{(j)}^+:= R^+\cap (\oplus_{\alpha\in \Pi_j}\mathbb{Z}\alpha)$ the set of positive roots of the root system of $\Pi_j$,  by $\theta_{(j)}$ the highest root in $R_{(j)}^+$, and  by $\alpha^{\rm ad}_{(j)}$ the (unique) simple root  in $\Pi_j$ adjacent to $\alpha_i$. Given $\beta\in R_{(j)}^+$, we write    $\beta= b_{\beta}\alpha^{\rm ad}_{(j)}+\sum_{\alpha\in \Pi_j\setminus\{\alpha^{\rm ad}_{(j)}\}} a_\alpha \alpha$, and denote by   $(\diamond)$ the property:
   $$(\diamond):\qquad \forall 1\leq j\leq k, \,\, \langle \alpha^{\rm ad}_{(j)}, \alpha_i^\vee\rangle =-1 \quad \mbox{and}\quad  b_{\theta_{(j)}}=1.
   $$
   Observe that any positive root $\beta$ in $R_P^+$ with $\langle \beta, \alpha_i^\vee\rangle \neq 0$ only if $\beta$ is in    the root subsystem spanned by $\Delta(\alpha_i)\setminus \alpha_i$. It suffices to look into such positive roots. Moreover, any such  root $\beta$ must be in   $R_{(j)}^+$ for some unique  $1\leq j\leq k$, and hence $\beta\leq \theta_{(j)}$ with respect to the   partial order $\leq $ given by $(c_1,\cdots, c_m)\leq (d_1, \cdots, d_m)$ iff $c_r\leq d_r$ for all $1\leq r\leq m$.   In particular, we have $0\leq b_\beta\leq b_{\theta_{(j)}}$. Thus if property $(\diamond)$ holds, then
     $$\langle \beta, \alpha_i^\vee\rangle =\langle b_\beta \alpha_{(j)}^{\rm ad}, \alpha_i^\vee\rangle=-b_{\beta}\in\{0, -1\}.$$
  It follows from the uniqueness in Lemma \ref{uniqueness} that $\lambda_B=\gamma^\vee$ for $\gamma=\alpha_i$. That is,   claim a) holds. Since
   $w_Pw_{P'}(\alpha_i)=\alpha_i+\sum_{\beta\in \Delta_P}\mathbb{Z}\beta$  and it is a root in $R$, it follows that $w_Pw_{P'}(\alpha_i)\in R^+$. Hence, $\ell(w_Pw_{P'}s_i)=\ell(w_Pw_{P'})+1$. That is, claim b) holds as well.
   We now check the property $(\diamond)$ across the Lie types.

   If $G$  is of Lie type $ADE$, then the property $(\diamond)$ obviously holds, and hence we are done. (Here we note that $\Pi_j$ can be at most of type $E_6, E_7$ but not of type $E_8$ since $\alpha_i\notin \Pi_j$.)

   For   $G$    of type $BCFG$, we just need to check the cases when the property $(\diamond)$ does not hold. We notice that $\{\alpha_i\}\cup \Pi_j$ is not of type $A$ for at most one $j$. If such $j$ exists, we say   $j=k$ without loss of generality.
   Write $\lambda_B=\alpha_i^\vee+\sum_{j=1}^k \lambda_j$ where $\lambda_j\in Q_{(j)}^\vee$.
   Then $\gamma_j^\vee:=\alpha_i^\vee+\lambda_j$ satisfies  $\gamma_j^\vee\in \alpha_i^\vee+ Q^\vee_{(j)}$ and $\langle \beta, \gamma_j^\vee\rangle=\langle \beta, \lambda_B\rangle \in \{0, -1\}$ for any  $\beta\in R^+_{(j)}$. It follows from the uniqueness in Lemma \ref{uniqueness} that $\lambda_j=0$ for $j<k$, in which case  $\{\alpha_i\}\cup \Pi_j$ is of type $A$. Thus   $\lambda_B=\gamma_k^\vee$.

 Suppose that $G$ is of type $B$ and the property $(\diamond)$ does not hold. Then $\Pi_k\cup\{\alpha_i\}$ must   also be of type $B$ with $\alpha_i$ being the (unique) short simple root at an end of the associated Dynkin diagram.  Therefore we can rename
 $\Pi_k\cup\{\alpha_i\}=\{\beta_1, \cdots, \beta_{m+1}\}$ such that it is  of type $B_{m+1}$ in the standard way and  $\alpha_i=\beta_{m+1}$ is the short simple root.
   For any $\beta\in R_{(k)}^+$,  $\langle \beta, \beta_m^\vee+\beta_{m+1}^\vee\rangle$ is equal to $0$ if $\beta$ has no nonzero coefficient at $\beta_{m-1}$, or     equal to  $\langle \beta_{m-1}, \beta_m^\vee+\beta_{m+1}^\vee\rangle=-1$ otherwise (where we use the property of   $\Pi_k$ being of type $A_m$). By the uniqueness in Lemma \ref{uniqueness}, we have
  $\lambda_B=\beta_m^\vee+\beta_{m+1}^\vee=s_{\beta_{m+1}}(\beta_m^\vee)$. Hence, $\lambda_B=\gamma^\vee$ for $\gamma= s_{\beta_{m+1}}(\beta_m)\in R^+$. Consequently, $s_\gamma=s_{\beta{m+1}}s_{\beta_m}s_{\beta_{m+1}}$ and $w_k=s_{\beta_2}\cdots s_{\beta_m}s_{\beta_1}\cdots s_{\beta_{m-1}}$ (where we mean $w_k=\mbox{id}$ if $m=1$).
  Clearly, $\ell(w_ks_\gamma)=\ell(w_k)+\ell(s_\gamma)=\ell(w_k)+2\langle \rho, \gamma^\vee\rangle -1$ by a direct calculation. Therefore the claims hold by noting $w_j=\rm id$ for $j=1, \cdots, k-1$. 

Suppose that $G$ is of type $G_2$ and the property $(\diamond)$ does not hold. Then $\Pi_k\cup\{\alpha_i\}=\{\beta_1,   \beta_{2}\}$,  and   $\alpha_i=\beta_{2}$ is the   short   root.
Everything is the same as the type  $B$ case with $m=1$ therein, except that $w_k= s_{\beta_1}$ for $G_2$. Clearly, the  claims hold.

Suppose that $G$ is of type $C$ and the property $(\diamond)$ does not hold. Then $\Pi_k\cup\{\alpha_i\}$ must     be of type $C$ with $\alpha_i$ being a short simple root at an end of the associated Dynkin diagram.  Thus we can rename
 $\Pi_k\cup\{\alpha_i\}=\{\beta_1, \cdots, \beta_{m+1}\}$ such that it is  of type $C_{m+1}$ in the standard way, $\alpha_i=\beta_{m+1}$ is a short simple root and $\beta_1$ is the long simple root. Note  $\langle \alpha, \sum_{j=1}^{m+1} \beta_j^\vee\rangle=0$ for any $\alpha\in \Pi_k$. By the uniqueness, we have
   $\lambda_B=\sum_{j=1}^{m+1} \beta_j^\vee= (s_{\beta_{m+1}}\cdots s_{\beta_2}(\beta_1) )^\vee$, and hence  $w_k={\rm id}$. The claims follow obviously.

 Suppose that $G$ is of type $F_4$. We may assume that $\Delta=\{\alpha_1, \alpha_2, \alpha_3, \alpha_4\}$ is already in the standard way such that $\alpha_1, \alpha_2$ are the long simple roots, and  $\alpha_1$, $\alpha_4$ are the two ends of the Dynkin diagram. There are $16$ possibilities of $\Delta_P$  in total, which can be easily discussed as follows.
 \begin{enumerate}
   \item Case $\alpha_1\notin \Delta_P$. If $\alpha_i\in\{\alpha_1, \alpha_2\}$, then the property $(\diamond)$ holds, and hence we are done. Otherwise,  $\alpha_i\in \{\alpha_3, \alpha_4\}$, and this is reduced to type $C$ case, so that we are done again.
    \item Case $\alpha_1\in \Delta_P$ and $\alpha_2\notin \Delta_P$.  Then either the property $(\diamond)$ holds or it is reduced to type $C$ case.
     \item Case $\alpha_1\in \Delta_P$ and $\alpha_2\in \Delta_P$. If $\alpha_3\notin \Delta_P$, then either the property $(\diamond)$ holds or it  is reduced to type $B$ case.   If $\alpha_3\in \Delta_P$,
 then $\alpha_i=\alpha_4$.  By computing $\langle \beta, \sum_{j=2}^{4} \alpha_j^\vee\rangle$ for all $\beta\in R_P^+$, we conclude $\lambda_B=\sum_{j=2}^{4} \alpha_j^\vee=s_{{4}}  s_{3}(\alpha_2^\vee)=\gamma^\vee$ for $\gamma=s_4s_3(\alpha_2)\in R^+$,  by the uniqueness of $\lambda_B$. Consequently,    $w_k=s_{1}s_{2}s_{3}s_{2}s_{1}$, and hence the claims follow by direct calculations.
  \end{enumerate}
  The nonvanishing of $c_{s_i, u}^{(0^P)^\vee, \lambda_P}$ implies that $\ell(u)=\deg (q_i)+\ell(0^P)-\ell(s_i)=n_i-1$.
 \end{proof}

\begin{remark}
 Given $\gamma\in R^+$, the property $\ell(\gamma)=2\langle\rho, \gamma^\vee\rangle -1$ always holds whenever the Dynkin diagram of $G$ is simply laced (i.e., of Lie type $ADE$) \cite[Lemma 3.2]{mare}. In general, the property requires a nontrivial condition \cite[Lemma 3.8]{LL01}.
\end{remark}

\begin{lemma}\label{construction}
For each $i\in I^P,$ there is a cycle $\Xi_i$ of  length $n_i$ in $D(X)$ through the fixed vertex $\sigma_{0^P}$.
\end{lemma}
\begin{proof}
By Lemma \ref{qiterm}, there exists $u\in W^P$ of length $n_i-1$ such that $\sigma_{0^P}$ occurs in $\sigma_{s_i}\star \sigma_u$, and hence occurs in the expansion $[c_1(X)](\sigma_{u})$. Thus there is a   path $\mathrm{PATH}(u: 0^P)$ of length 1 from the vertex $\sigma_u$ to the vertex  $\sigma_{0^P}$.
Now take a reduced decomposition $u=s_{j_1}s_{j_2}\cdots s_{j_{n_i-1}}$, set $v_m:=s_{j_m}s_{j_{m+1}}\cdots s_{j_{n_i-1}}$ for $1\leq m\leq n_i-1$, and denote $v_0=v_{n_i}:=0^P$. Then
 we have $v_m\in W^P$ and  $v_{m+1}^{-1}(\alpha_{j_m})\in R^+\setminus R_P^+$ for all $1\leq m\leq n_i-1$ by Lemma \ref{reducedword}. Consequently, $\sigma_{v_m}$ occurs in $[c_1(X)](\sigma_{v_{m+1}})$ by the (quantum) Chevalley formula, resulting in a path $\mathrm{PATH}(v_{m+1}: v_m)$ of length $1$ for all $1\leq m\leq n_i-1$. Clearly, the join of paths
  $$\mathrm{PATH}(v_{n_i}: v_{n_i-1})\sqcup_{v_{n_i-1}}   \cdots \sqcup_{v_{2}}  \mathrm{PATH}(v_{2}: v_{1}) \sqcup_{v_1}\mathrm{PATH}(u: 0^P)$$ is a cycle of length $n_i$ through $\sigma_{0^P}$, by noting that $v_1=u$ and $v_{n_i}=0^P$.
\end{proof}

Recall that   $h(D(X))$ is the greatest common divisor (g.c.d) of the lengths of all cycles of $D(X)$. Therefore it divides  the g.c.d  of the lengths of the cycles $\{\Xi_i~|~ i\in I^P\}$ through $\sigma_{0^P}$. Namely,   $h(D(X))$ divides g.c.d.$\{n_i~|~i\in I^P\}=r$ by Lemma \ref{construction} and the definition of the Fano index $r$. Since $h(D(X))=h(M(X)_q)=h(M(X))$,   $h(M(X))$ divides $r$.  This fact together with Lemma \ref{coro-divide-p} implies
\begin{corollary}\label{equal2}
The imprimitivity index $h(M(X))$ of the irreducible matrix $M(X)$ is equal to the Fano index $r$.
\end{corollary}

\begin{theorem}
Homogeneous spaces $X=G/P$ satisfy Conjecture $\mo.$
\end{theorem}
\begin{proof}
Since $M(X)$ is irreducible, Condition $(1)$ and $(2)$ are satisfied by Proposition \ref{Perron-Frob}. Condition $(3)$ follows from Corollary \ref{equal2} and Proposition \ref{prop-angle}.
\end{proof}

\end{document}